\documentclass[amstex,10pt,reqno]{amsart}
\usepackage{amsmath,amsfonts,amssymb,amsthm,enumerate,hyperref,multicol, graphicx, wrapfig, pdfsync}
\newtheorem{lemma}{Lemma}
\newtheorem{thm}{Theorem}

\newtheorem{prop}{Proposition}
\newtheorem{remark}{Remark}
\numberwithin{equation}{section}
\begin{document}

\leftline{ \scriptsize \it}
\title[]
{Approximation properties of complex genuine $\alpha-$Bernstein-Durrmeyer operators}
\maketitle

\begin{center}
{\bf  Meenu Goyal}
\vskip0.2in
School of Mathematics\\
Thapar Institute of Engineering and Technology, Patiala\\
Patiala-147004, India\\
\vskip0.2in
meenu\_rani@thapar.edu
\end{center}

\begin{abstract}
Herein we propose a complex form of a genuine Bernstein-Durrmeyer type operators depending on a non-negative real parameter $\alpha.$ We present the quantitative upper bound, Voronovskaja type result and exact order of approximation for these operators and for their derivatives attached to analytic functions on compact disks. These results validate the extension of approximation properties of complex genuine $\alpha-$Bernstein-Durrmeyer type operators from real intervals to compact disks in the complex plane.
\end{abstract}

\section{Introduction}
In \cite{CTLX}, Chen et al. introduced a new generalization of Bernstein operators depending on a non-negative real parameter $\alpha$, which are having some remarkable properties, given by
\begin{eqnarray}\label{g1}
\mathcal{B}_n^\alpha (f;x)= \sum_{k=0}^{n} q_{n,k}^{\alpha}(x) f\left(\frac{k}{n}\right),
\end{eqnarray}
for any function $f(x)$ defined on $[0,1], \, n\in\mathbb{N}.$ \\
The $\alpha-$Bernstein polynomials $q_{n,k}^{\alpha}(x)$ of degree $n$ is defined by $q_{1,0}^{\alpha}(x)=1-x, \, q_{1,1}^{\alpha}(x)=x$\\ and
\begin{eqnarray*}
q_{n,k}^{\alpha}(x)= \left[{n-2 \choose k} (1-\alpha)x + {n-2 \choose k-2} (1-\alpha)(1-x)+ {n \choose k} \alpha x(1-x)\right] x^{k-1}(1-x)^{n-k-1}
\end{eqnarray*}
for $n\geq 2, \, x\in [0,1].$ $\mathcal{B}_n^\alpha(f;x)$ are positive operators for $0\leq \alpha\leq 1.$\\ These operators have many useful properties such as preservation of  linear (or constant) functions, monotonicity and convexity.
To approximate the Lebesgue integrable functions by $\alpha-$Bernstein operators, Acar et. al \cite{AMM} introduced a genuine $\alpha-$Bernstein-Durrmeyer operators as follows:
\begin{eqnarray}\label{m1}
G_n^\alpha (f;x)= q_{n,0}^{\alpha}(x) f(0)+q_{n,n}^{\alpha}(x) f(1)+(n-1) \sum_{k=1}^{n-1} q_{n,k}^{\alpha}(x) \int_0^1 p_{n-2,k-1}(t) f(t) dt,
\end{eqnarray}
where $p_{n,k}(t)= {n\choose k} t^k (1-t)^{n-k}.$
Authors studied local approximation, error estimation in terms of Ditzian-Totik modulus of smoothness and the convergence of these operators to certain functions by illustrative graphics.

Very recently, \c{C}etin \cite{NC} introduced the complex form of the operators defined in $(\ref{g1}).$ These operators exhibit the approximation properties i.e. upper estimate, asymptotic formula, exact order of convergence and some shape preserving properties in the complex domain. These properties for complex Bernstein polynomials in compact disks were initially investigated by Lorentz \cite{GGL}. Also, in \cite{SG}, a very useful book by Gal and references therein, the overconvergence properties of the well known complex operators i.e. Bernstein-Stancu operators, $q-$Bernstein polynomials, Kantorovich operators etc. have been discussed. He obtained quantitative estimates for these operators. He put in evidence the overconvergence phenomenon for several operators, namely the extensions of approximation properties with exact quantitative estimates, from the real interval to compact disks in the complex plane. Several type of complex operators were studied by a number of authors (see \cite{AG1}-\cite{VG2}, \cite{GG}- \cite{NIM1}).

Motivated from the real case and overconvergence properties of above work on complex operators, we consider the complex form of the operators $(\ref{m1}),$ for $n\in \mathbb{N}, \, \alpha\in [0,1],\, z\in \mathbb{C}$ and $f$ is complex valued analytic function in  an open disk $D_R=\{z\in \mathbb{C}, |z|<R\},\ R>1,$ as follows:
\begin{eqnarray}\label{n1}
G_n^\alpha (f;z)= q_{n,0}^{\alpha}(z) f(0)+q_{n,n}^{\alpha}(z) f(1)+(n-1) \sum_{k=1}^{n-1} q_{n,k}^{\alpha}(z) \int_0^1 p_{n-2,k-1}(t) f(t) dt.
\end{eqnarray}
The aim of this paper is to present the approximation properties of complex genuine $\alpha-$Bernstein-Durrmeyer operators. Here an upper estimate, Voronovskaja type result, exact order and simultaneous approximation are obtained for the order of approximation of these operators attached to analytic functions on a certain compact disk which give significant contribution in the theory of uniform convergence without depending on the parameter $\alpha.$

\section{auxiliary results}
Let $e_p(z)=z^p, p\in \mathbb{N}\cup \{0\}, z\in \mathbb{C}.$ From the definition of our operators, it can be easily shown that $\displaystyle G_n^\alpha (e_0;z)=1, \, G_n^\alpha (e_1;z)=z.$\\
We first mention the recurrence relation for moments:
\begin{prop}\label{p1}
For all $p\in \mathbb{N}\cup \{0\}, n \in \mathbb{N},\,z\in \mathbb{C}$ and $\alpha\in [0,1],$ we have
\begin{eqnarray*}
G_n^\alpha (e_{p+1};z)= \frac{z(1-z)}{n+p} (G_n^\alpha (e_p;z))^\prime +\frac{nz+p}{n+p}G_n^\alpha (e_p;z) -\frac{z}{n+p} S_1(e_p;z) +\frac{1-z}{n+p} S_2(e_p;z),
\end{eqnarray*}
where $\displaystyle S_1(e_p;z)=(n-1)\sum_{k=0}^{n-1} p_{n-2,k}(z) (1-\alpha) (1-z) \int_0^1 p_{n-2,k-1}(t) t^p dt$ \\and $\displaystyle S_2(e_p;z)= (n-1)\sum_{k=2}^{n} p_{n-2,k-2}(z) (1-\alpha)z \int_0^1 p_{n-2,k-1}(t) t^p dt.$
\end{prop}
\begin{proof}
If $p=0,$ then the recurrence relation is immediate from $G_n^\alpha (e_0;z)=1,\,\,G_n^\alpha (e_1;z)=e_1(z),\,S_1(e_0;z)= (1-z)(1-\alpha)$ and $S_2(e_0;z)= z(1-\alpha).$ So, let us suppose that $p\geq 1.$\\
Denote $\displaystyle I= \int_0^1 p_{n-2,k-1}(t) t^p dt = {n-2 \choose k-1} B(k+p,n-k),$ from the formula of definition, we can write
\begin{eqnarray*}
G_n^\alpha (e_p;z)&=& q_{n,n}^\alpha(z)+(n-1) \sum_{k=1}^{n-1} q_{n,k}^\alpha(z) * I\\
&=& (n-1)\sum_{k=1}^{n-1} (1-\alpha)p_{n-2,k}(z)(1-z) * I + (n-1)\sum_{k=1}^{n} (1-\alpha)p_{n-2,k-2}(z) z * I\\&&+ (n-1)\sum_{k=1}^{n} \alpha\, p_{n,k}(z) * I:= S_1(e_p;z)+S_2(e_p;z)+S_3(e_p;z)\,(say)\\
\frac{d}{dz} (S_1(e_p;z)) &=& (n-1)\sum_{k=1}^{n-1} (1-\alpha)\left(p_{n-2,k}(z)(1-z)\right)^\prime * I \\
&=& \frac{(n-1) (1-\alpha)}{z(1-z)} \sum_{k=1}^{n-1} p_{n-2,k}(z)(1-z) (k-(n-1)z) * I\\
&=& -\frac{n-1}{1-z} S_1(e_p;z) +\frac{(1-\alpha)(n-1)}{z(1-z)} \sum_{k=1}^{n-1}  p_{n-2,k}(z)(1-z)\, k*I.
\end{eqnarray*}
Now, from the formula for beta function $p\, B(p,q)= (p+q)\, B(p+1,q) \,\,\forall \,\, p,q>0,$\\
we have $(k+p)\, B(k+p,n-k)= (n+p)\, B(k+p+1,n-k)\,$\\
$\Rightarrow k \,B(k+p,n-k)= (n+p)\, B(k+p+1,n-k)-p\, B(k+p,n-k).$\\
Therefore, we obtain
\begin{eqnarray}\label{m2}
(S_1(e_p;z))^\prime &=& -\frac{n-1}{1-z} S_1(e_p;z) +\frac{(1-\alpha)(n-1)}{z(1-z)} \sum_{k=1}^{n-1}  p_{n-2,k}(z)(1-z) {n-2 \choose k-1}\nonumber\\
&&\times\left((n+p)\, B(k+p+1,n-k)-p\, B(k+p,n-k)\right)\nonumber\\
z(1-z) (S_1(e_p;z))^\prime &=& (n+p)S_1(e_{p+1};z)-((n-1)z+p)S_1(e_p;z).
\end{eqnarray}
Similarly,
\begin{eqnarray}\label{m3}
z(1-z) (S_2(e_p;z))^\prime =(n+p)S_2(e_{p+1};z)-((n-1)z+p+1)S_2(e_p;z)
\end{eqnarray}
and
\begin{eqnarray}\label{m4}
z(1-z) (S_3(e_p;z))^\prime = (n+p)S_3(e_{p+1};z)-(nz+p)S_3(e_p;z).
\end{eqnarray}
By combining equations $(\ref{m2})-(\ref{m4}),$ we get the required result.
\end{proof}

\begin{remark}
By using proposition $\ref{p1},$ we get
\begin{eqnarray*}
G_n^\alpha (e_2;z)=z^2+\frac{2 z(1-z)}{n+1}\left(1+\frac{1-\alpha}{n}\right).
\end{eqnarray*}
\end{remark}

\begin{lemma} \label{l2}
(i) For $n\in \mathbb{N},\,p\in \mathbb{N}\cup \{0\}$ and $\alpha \in [0,1],$ we have $G_n^\alpha (e_p;1)=1.$\\
(ii) For $n,p\in \mathbb{N},\,z\in \mathbb{C}$ and $\alpha \in [0,1],$ we have
\begin{eqnarray*}
G_n^\alpha (e_p;z)&=& \frac{(n-1)!}{(n-1+p)!} \left((1-\alpha) \sum_{s=0}^{n-1}{n-1 \choose s} z^s \bigtriangleup_1^s F_p(0)+\alpha \sum_{s=0}^{n}{n\choose s} z^s \bigtriangleup_1^s E_p(0)\right)\\&=& \frac{(n-1)!}{(n-1+p)!} \left((1-\alpha) \sum_{s=0}^{\min\{n-1,p\}}{n-1 \choose s} z^s \bigtriangleup_1^s F_p(0)+\alpha \sum_{s=0}^{\min\{n,p\}}{n\choose s} z^s \bigtriangleup_1^s E_p(0)\right),
\end{eqnarray*}
where $\displaystyle F_p(k)= \left(1-\frac{k}{n-1}\right)E_p(k)+\frac{k}{n-1}E_p(k+1),\, E_p(k)= \prod_{j=0}^{p-1}(k+j)\,\, \forall \, k\geq 0$ and $\displaystyle\bigtriangleup_1^s F_p(0)= \sum_{k=0}^s (-1)^k {s \choose k} F_p(s-k), \, \bigtriangleup_1^s E_p(0)= \sum_{k=0}^s (-1)^k {s \choose k} E_p(s-k)$ and $\displaystyle\bigtriangleup_1^s F_p(0), \bigtriangleup_1^s E_p(0)\geq 0$ for all $s$ and $p.$
\end{lemma}
\begin{proof} (i) From the definition of the operators $(\ref{n1})$
\begin{eqnarray*}
G_n^\alpha (e_p;z)&=& q_{n,n}^{(\alpha)}(z)+(n-1) \sum_{k=1}^{n-1} q_{n,k}^{(\alpha)}(z) * I\\
G_n^\alpha (e_p;1)&=& q_{n,n}^{(\alpha)}(1)+(n-1) \sum_{k=1}^{n-1} q_{n,k}^{(\alpha)}(1) * I=1.
\end{eqnarray*}

(ii) Let us denote $\prod_{j=0}^{p-1}(\mu+j):=E_p(\mu).$ So, $I$ can be written as
\begin{eqnarray*}
I= {n-2 \choose k-1} B(k+p, n-k)&=& \frac{(n-2)!}{(k-1)!(n-k-1)!} * \frac{(k+p-1)!(n-k-1)!}{(n-1+p)!}\\
&=& \frac{(n-2)!}{(n-1+p)!}(k(k+1)(k+2)\dotsb (k+p-1)):=\frac{(n-2)!}{(n-1+p)!}E_p(k),
\end{eqnarray*}
It is clear that $E_p(\mu)$ and its derivative of any order are $\geq 0\,\, \forall\,\, \mu\geq 0,$ i.e. $\Rightarrow \triangle_1^k E_p(0)\geq 0\,\, \forall\,\, k, p.$

So, we can write
\begin{eqnarray*}
G_n^\alpha (e_p;z)&=& q_{n,n}^\alpha(z)+\frac{(n-1)!}{(n+p-1)!} \sum_{k=1}^{n-1} q_{n,k}^{(\alpha)}(z) * E_p(k)\\
&=& q_{n,n}^\alpha(z)+\frac{(n-1)!}{(n+p-1)!}\left( \sum_{k=0}^{n} q_{n,k}^{(\alpha)}(z) * E_p(k)-q_{n,0}^{(\alpha)}(z) E_p(0)-q_{n,n}^{(\alpha)}(z)E_p(n)\right)\\
&=& q_{n,n}^\alpha(z)-\frac{(n-1)!}{(n+p-1)!}q_{n,n}^\alpha(z) E_p(n)+\frac{(n-1)!}{(n+p-1)!}\sum_{k=0}^{n} q_{n,k}^{(\alpha)}(z) * E_p(k)\\
&=& \frac{(n-1)!}{(n+p-1)!}\sum_{k=0}^{n} q_{n,k}^{(\alpha)}(z) * E_p(k)= \frac{(n-1)!}{(n+p-1)!}\left((1-\alpha)\sum_{k=0}^{n-1} p_{n-2,k}(z)(1-z) * E_p(k)\right.\\&&\left.+(1-\alpha)\sum_{k=1}^{n-1} p_{n-2,k-1}(z)z * E_p(k+1)+\alpha \sum_{k=0}^{n} p_{n,k}(z) * E_p(k) \right)\\
&=& \frac{(n-1)!}{(n+p-1)!}\left[(1-\alpha)\sum_{k=0}^{n-1}{n-1 \choose k}\left\{ \left(1-\frac{k}{n-1}\right)E_p(k)+\frac{k}{n-1}E_p(k+1)\right\}z^k(1-z)^{n-k-1}\right.\\&&\left.+\alpha\sum_{k=0}^{n} p_{n,k}(z) * E_p(k) \right]\\
&=& \frac{(n-1)!}{(n+p-1)!}\left[(1-\alpha)\sum_{k=0}^{n-1}p_{n-1,k}(z) * F_p(k)+\alpha\sum_{k=0}^{n} p_{n,k}(z) * E_p(k)\right],
\end{eqnarray*}
where $\displaystyle F_p(k)= \left(1-\frac{k}{n-1}\right)E_p(k)+ \frac{k}{n-1}E_p(k+1).$
\begin{eqnarray*}
G_n^\alpha (e_p;z)&=& \frac{(n-1)!}{(n+p-1)!}\left((1-\alpha)\sum_{k=0}^{n-1}{n-1\choose k} z^k F_p(k)\sum_{j=0}^{n-1-k} (-1)^j{n-1-k \choose j} z^j\right.\\
&&\left.+\alpha\sum_{k=0}^{n} {n \choose k} z^k *E_p(k)\sum_{j=0}^{n-k} (-1)^j {n-k \choose j} z^j\right)\\
\end{eqnarray*}
Assume $j=s-k$ in both the summations of above equation, then we get
\begin{eqnarray*}
&=& \frac{(n-1)!}{(n+p-1)!}\left((1-\alpha)\sum_{s=0}^{n-1}{n-1\choose s} z^s \sum_{k=0}^{s} (-1)^{s-k}{s \choose k}F_p(k)\right.\\
&&\left.+\alpha\sum_{s=0}^{n} {n \choose s} z^s \sum_{k=0}^{s} (-1)^{s-k} {s \choose k} E_p(k)\right)\\
&=&\frac{(n-1)!}{(n+p-1)!}\left((1-\alpha)\sum_{s=0}^{n-1}{n-1\choose s} z^s \sum_{k=0}^{s} (-1)^{k}{s \choose k}F_p(s-k)\right.\\
&&\left.+\alpha\sum_{s=0}^{n} {n \choose s}z^s \sum_{k=0}^{s} (-1)^{k} {s \choose k} E_p(s-k)\right)\\
&=& \frac{(n-1)!}{(n+p-1)!}\left((1-\alpha)\sum_{s=0}^{\min\{n-1,p\}}{n-1\choose s} z^s \triangle_1^s F_p(0)+\alpha\sum_{s=0}^{\min\{n,p\}} {n \choose s}\sum_{q=0}^{s} \triangle_1^s E_p(0)\right).
\end{eqnarray*}
This completes the proof.
\end{proof}

\section{Main results}
The first main result of this section is the following upper estimate:
\subsection{Upper estimate}
\begin{thm}\label{t1}
(i) For all $p,n\in \mathbb{N}\cup \{0\},\, |z|\leq r$ and $\alpha \in [0,1],$ we have $|G_n^\alpha(e_p;z)|\leq r^p.$\\
(ii) Let $f(z)= \displaystyle\sum_{p=0}^\infty c_p z^p \, \forall\,|z|<R,\, 1\leq r<R.$ For all $|z|\leq r,\,n\in \mathbb{N}$ and $\alpha \in [0,1],$ we have
\begin{eqnarray*}
|G_n^\alpha (f;z)-f(z)|\leq \frac{C_r(f)}{n}, \,\,\mbox{where}\,\, C_r(f)= 2 \sum_{p=2}^\infty |c_p| p(p-1)r^p<\infty.
\end{eqnarray*}
\end{thm}
\begin{proof}
From Lemma \ref{l2} (i) and (ii), we have
\begin{eqnarray*}
G_n^\alpha(e_p;1)=\frac{(n-1)!}{(n-1+p)!} \left((1-\alpha) \sum_{s=0}^{\min\{n-1,p\}}{n-1 \choose s} \bigtriangleup_1^s F_p(0)+\alpha \sum_{s=0}^{\min\{n,p\}}{n\choose s} \bigtriangleup_1^s E_p(0)\right)=1,
\end{eqnarray*}
then
\begin{eqnarray*}
|G_n^\alpha (e_p;z)|&\leq& \frac{(n-1)!}{(n-1+p)!} \left((1-\alpha) \sum_{s=0}^{\min\{n-1,p\}}{n-1 \choose s} \bigtriangleup_1^s F_p(0)+\alpha \sum_{s=0}^{\min\{n,p\}}{n\choose s} \bigtriangleup_1^s E_p(0)\right) r^s\\
&\leq& \frac{(n-1)!}{(n-1+p)!} \left((1-\alpha) \sum_{s=0}^{\min\{n-1,p\}}{n-1 \choose s} \bigtriangleup_1^s F_p(0)+\alpha \sum_{s=0}^{\min\{n,p\}}{n\choose s} \bigtriangleup_1^s E_p(0)\right) r^p=r^p
\end{eqnarray*}
proves (i).\\
(ii) To find the upper estimate for the operators $G_n^\alpha(f;z),$ we first prove that $\displaystyle G_n^\alpha(f;z)= \sum_{p=0}^\infty c_p G_n^\alpha(e_p;z).$ For this let $f_m(z)= \displaystyle\sum_{p=0}^m c_p e_p(z), \, |z|\leq r, \, m\in \mathbb{N}.$\\ By linearity property of $G_n^\alpha(f;z),$ we can write $G_n^\alpha(f_m;z)= \displaystyle\sum_{p=0}^m c_p G_n^\alpha(e_p;z)\,\,\forall \,\, |z|\leq r.$\\ So, $\forall \,\, \,|z|\leq r, \, n\in \mathbb{N},$ it is sufficient to prove that $\displaystyle\lim_{m\rightarrow \infty}G_n^\alpha(f_m;z)= G_n^\alpha(f;z).$\\
Now, for all $|z|\leq r, \, r\geq 1,$ we have\\
$|G_n^\alpha (f_m;z)-G_n^\alpha (f;z)|$
\begin{eqnarray*}
&\leq& |f_m(0)-f(0)|*|(1-\alpha)(1-z)^{n-1}+\alpha (1-z)^n|+ |f_m(1)-f(1)|*|(1-\alpha)z^{n-1}+\alpha z^n|\\&&+(n-1)\sum_{k=1}^{n-1} q_{n,k}^\alpha (z) \int_0^1 p_{n-2,k-1}(t) *|f_m(t)-f(t)| dt\\
&\leq&  |f_m(0)-f(0)|* ((1-\alpha)(1+r)^{n-1}+\alpha (1+r)^n)+ |f_m(1)-f(1)|* ((1-\alpha)r^{n-1}+\alpha r^n) \\&&+(n-1)\sum_{k=1}^{n-1}\left\{{n-2 \choose k} (1-\alpha)r+{n-2 \choose k-2} (1-\alpha)(1+r)+{n \choose k} \alpha r(1+r)\right\} r^{k-1}(1+r)^{n-k-1}\\&&\times \int_0^1 p_{n-2,k-1}(t)* |f_m(t)-f(t)| dt\leq C_{r,n}^\alpha \|f_m-f\|_r,
\end{eqnarray*}
where
\begin{eqnarray*}
C_{r,n}^\alpha &=& \alpha (r^n+(1+r)^n)+(1-\alpha)(r^{n-1}+(1+r)^{n-1})\\&&+(n-1)\displaystyle\sum_{k=1}^{n-1}\left \{(1-\alpha){n-2 \choose k} r +(1-\alpha){n-2 \choose k-2} (1+r)+\alpha{n\choose k} r(1+r)\right\}r^{k-1}(1+r)^{n-k-1}\\&&\times \int_0^1 p_{n-2,k-1}(t) dt.
\end{eqnarray*}
Therefore, we get
\begin{eqnarray*}
|G_n^\alpha (f;z)-f(z)|\leq \sum_{p=0}^\infty |c_p| |G_n(e_p;z)-e_p(z)|= \sum_{p=2}^\infty |c_p| |G_n(e_p;z)-e_p(z)|
\end{eqnarray*}
as $G_n^\alpha (e_0;z)=e_0(z)=1, \, G_n(e_1;z)=e_1(z)=z.$\\

Now, we have two cases: (A) $2\leq p\leq n,$ \,\, (B) $p>n.$\\
Case (A): From Lemma $\ref{l2},$ we obtain
\begin{eqnarray*}
G_n^\alpha (e_p;z)-e_p(z)&=& \frac{(n-1)!}{(n-1+p)!} \left((1-\alpha) \sum_{s=0}^{p-1}{n-1 \choose s} z^s \bigtriangleup_1^s F_p(0)+\alpha \sum_{s=0}^{p-1}{n\choose s} z^s \bigtriangleup_1^s E_p(0)\right)\\&&+ \left(\frac{(n-1)!}{(n-1+p)!}\left\{(1-\alpha) {n-1 \choose p} \bigtriangleup_1^p F_p(0)+\alpha {n\choose p} \bigtriangleup_1^p E_p(0)\right\}-1\right)z^p\\
|G_n^\alpha (e_p;z)-e_p(z)|&\leq & r^p\left(1-\frac{(n-1)!}{(n-1+p)!}\left\{(1-\alpha) {n-1 \choose p} \bigtriangleup_1^p F_p(0)+\alpha {n\choose p} \bigtriangleup_1^p E_p(0)\right\}\right)\\&&+\left\vert \frac{(n-1)!}{(n-1+p)!} \left((1-\alpha) \sum_{s=0}^{p-1}{n-1 \choose s} z^s \bigtriangleup_1^s F_p(0)+\alpha \sum_{s=0}^{p-1}{n\choose s} z^s \bigtriangleup_1^s E_p(0)\right)\right\vert\\
&\leq& r^p\left(1-\frac{(n-1)!}{(n-1+p)!}\left\{(1-\alpha) {n-1 \choose p} \bigtriangleup_1^p F_p(0)+\alpha {n\choose p} \bigtriangleup_1^p E_p(0)\right\}\right)\\&&+ \frac{(n-1)!}{(n-1+p)!}\left\vert \left((1-\alpha) \sum_{s=0}^{p}{n-1 \choose s} z^s \bigtriangleup_1^s F_p(0)+\alpha \sum_{s=0}^{p}{n\choose s} z^s \bigtriangleup_1^s E_p(0)\right)\right.\\&&\left.-(1-\alpha){n-1\choose p} z^p \bigtriangleup_1^p F_p(0)-\alpha {n\choose p} z^p \bigtriangleup_1^p E_p(0)\right\vert\\
&\leq& r^p\left(1-\frac{(n-1)!}{(n-1+p)!}\left\{(1-\alpha) {n-1 \choose p} \bigtriangleup_1^p F_p(0)+\alpha {n\choose p} \bigtriangleup_1^p E_p(0)\right\}\right)\\&&+ \frac{(n-1)!}{(n-1+p)!}r^p \left\vert (1-\alpha){n-1\choose p} \bigtriangleup_1^p F_p(0)+\alpha {n\choose p} \bigtriangleup_1^p E_p(0)\right.\\&&\left.- \left((1-\alpha) \sum_{s=0}^{p}{n-1 \choose s} \bigtriangleup_1^s F_p(0)+\alpha \sum_{s=0}^{p}{n\choose s} \bigtriangleup_1^s E_p(0)\right)\right\vert\\
&\leq& 2 r^p \left(1-\frac{(n-1)!}{(n-1+p)!}\left\{(1-\alpha) {n-1 \choose p} \bigtriangleup_1^p F_p(0)+\alpha {n\choose p} \bigtriangleup_1^p E_p(0)\right\}\right).
\end{eqnarray*}
Now, by using the relationship between derivatives and differences of a function (see \cite{CTLX}, pg-251), we get \\
\noindent
$\displaystyle\frac{(n-1)!}{(n-1+p)!}\left\{(1-\alpha) {n-1 \choose p} \bigtriangleup_1^p F_p(0)+\alpha {n\choose p} \bigtriangleup_1^p E_p(0)\right\}$
\begin{eqnarray*}
&=&\frac{(n-1)!}{(n-1+p)!}\left\{(1-\alpha) {n-1 \choose p} \left(1+\frac{p}{n-1}\right) p!+\alpha {n\choose p} p!\right\}\\
&=&\frac{(n-1)!}{(n-1+p)!}\left\{(1-\alpha)  \left(\frac{(n-2)!(n-1+p)}{(n-1-p)!}\right) +\alpha \left(\frac{n!}{(n-p)!}\right)\right\}\\
&=&\prod_{j=1}^{p-1}\left((1-\alpha) \frac{n+(j-1)-p}{n+j-1}+\alpha \frac{n+j-p}{n+j}\right).
\end{eqnarray*}
Since $\displaystyle 1-\prod_{j=1}^k a_j\leq \sum_{j=1}^k (1-a_j),\,\, 0\leq a_j \leq 1,\,\, j=1,2, \dotsb k,$\\
then $\displaystyle1-\frac{(n-1)!}{(n-1+p)!}\left\{(1-\alpha) {n-1 \choose p} \bigtriangleup_1^p F_p(0)+\alpha {n\choose p} \bigtriangleup_1^p E_p(0)\right\}$
\begin{eqnarray*}
&=& 1-\prod_{j=1}^{p-1}\left((1-\alpha) \frac{n+(j-1)-p}{n+j-1}+\alpha \frac{n+j-p}{n+j}\right)\\
&\leq&  \sum_{j=1}^{p-1} 1- \left((1-\alpha)\frac{n+(j-1)-p}{n+j-1}+\alpha\frac{n+j-p}{n+j}\right)\\
&=& p \sum_{j=1}^{p-1} \left(\frac{(1-\alpha)}{n+j-1}+\frac{\alpha}{n+j}\right)\leq p(p-1) \left(\frac{(1-\alpha)}{n}+\frac{\alpha}{n+1}\right)\\
\end{eqnarray*}
$\displaystyle\Rightarrow |G_n^\alpha (e_p;z)-e_p(z)|\leq 2 p(p-1)\, r^p \left(\frac{(1-\alpha)}{n}+\frac{\alpha}{n+1}\right)\leq \frac{2 p(p-1)}{n} r^p.$\\
Case(B): By $(i)$ and $p>n,$ we get\\
\begin{eqnarray*}
|G_n^\alpha (e_p;z)-e_p(z)|\leq |G_n^\alpha (e_p;z)|+|e_p(z)|\leq 2 r^p\leq \frac{2 p(p-1)}{n} r^p.
\end{eqnarray*}
Finally, by combining both the cases (A) and (B), $\forall\,\, p, n\in \mathbb{N}$, we obtain
\begin{eqnarray*}
|G_n^\alpha (e_p;z)-e_p(z)|&\leq& \frac{2 p(p-1)}{n} r^p\\
\Rightarrow |G_n^\alpha (f;z)-f(z)|&\leq& \sum_{p=2}^\infty |c_p| |G_n^\alpha (e_p;z)-e_p(z)|\leq  \frac{2}{n} \sum_{p=2}^\infty |c_p|p(p-1)r^p,
\end{eqnarray*}
which gives the desired result.
\end{proof}

Next, we have the following qualitative asymptotic formula for complex genuine $\alpha-$Bernstein-Durrmeyer operators $G_n^\alpha(f;z).$
\subsection{Voronovskaja type result}
\begin{thm}\label{t2}
Let $1\leq r <R$ and suppose that $f(z)= \displaystyle\sum_{p=0}^\infty c_p z^p, \, \forall \,\, |z|< R.$ Then\, $\forall\, n\in \mathbb{N}, \alpha\in [0,1]$ and $|z|\leq r,$ we have
\begin{eqnarray*}
\lim_{n\rightarrow \infty} n(G_n^\alpha (f;z)-f(z))=  z(1-z) f^{\prime\prime}(z), \,\, \mbox{uniformly in}\,\,  \overline{D}_r.
\end{eqnarray*}
\begin{proof}
From Acar et al. \cite{AMM}, it is known that
\begin{eqnarray*}
\lim_{n\rightarrow\infty} \left(n\left(G_n^\alpha (f;z)-f(z)\right)-\frac{(n+1-\alpha)}{n+1}z(1-z) f^{\prime\prime}(z)\right)=0,\,\,\forall \,\, x\in [0,1].
\end{eqnarray*}
According to Vitali's theorem stated in \cite{SG}, i.e. if a sequence $\left\{f_n\right\}_{n\in \mathbb{N}}$ of analytic functions in $D_R$ is bounded in each $\overline{D}_r,$ then it is uniformly convergent in $\overline{D}_r.$\\ So, it is sufficient to show that the sequence $\left\{n\left(G_n^\alpha (f;z)-f(z)\right)-\frac{(n+1-\alpha)}{n+1}z(1-z) f^{\prime\prime}(z)\right\}_{n\in \mathbb{N}}$ is bounded in each $\overline{D}_r.$\\
 By using Theorem $\ref{t1},$ we get
\begin{eqnarray*}
\left\vert n\left(G_n^\alpha (f;z)-f(z)\right)-\frac{(n+1-\alpha)}{n+1}z(1-z) f^{\prime\prime}(z)\right\vert\leq  C_r(f)+\frac{(n+1-\alpha)}{n+1}r(1+r) \left\Vert f^{\prime\prime}\right\Vert_r \leq  C_r(f)+r(1+r) \left\Vert f^{\prime\prime}\right\Vert_r,
 \end{eqnarray*}
for all $z\in \overline{D}_r,\,1\leq r < R$ and $\alpha\in [0,1],$ where $C_r(f)$ is the same constant as defined in Theorem $\ref{t1},$ which proves the theorem.
\end{proof}
\end{thm}

Finally, we find the exact order of approximation by using the above asymptotic result.
\subsection{Exact order of approximation}
\begin{thm}
Let $\alpha\in [0,1], f:D_R \rightarrow \mathbb{C}$ is analytic in $D_R, R>1$ with $f(z)=\displaystyle\sum_{p=0}^\infty c_p z^p.$ If $f$ is not a polynomial of degree $\leq 1,$ then for all $1\leq r <R,$ we have
\begin{eqnarray*}
\left\Vert G_n^\alpha(f)-f\right\Vert_r \thicksim \frac{1}{n},\, n\in \mathbb{N}
\end{eqnarray*}
in $\overline{D}_r,$ where the constants in the equivalence depend on $f$ and $r.$
\end{thm}
\begin{proof}
By using Theorem $\ref{t2},$ there exist constants $0<C_1, C_2<\infty$ independent of $n$ such that
\begin{eqnarray*}
C_1\leq n \left\Vert G_n^\alpha(f)-f\right\Vert_r\leq C_2,
\Rightarrow \,\, \frac{C_1}{n}\leq  \left\Vert G_n^\alpha(f)-f\right\Vert_r\leq \frac{C_2}{n} \,\,\mbox{holds in}\,\, \overline{D}_r.
\end{eqnarray*}
Thus, we get the required result.
\end{proof}

In the following result, we find the exact order of approximation for the derivatives of our operators $G_n^\alpha(f;z):$
\subsection{Simultaneous approximation}
\begin{thm}
Let $\alpha\in [0,1], f:D_R \rightarrow \mathbb{C}$ is analytic in $D_R, R>1$ with $f(z)=\displaystyle\sum_{p=0}^\infty c_p z^p.$ If $f$ is not a polynomial of degree $\leq \max\{1,l-1\},$ then for all $1\leq r <r_1<R$ and $l\in \mathbb{N},$ we have
\begin{eqnarray*}
\left\Vert G_n^{\alpha (l)}(f)-f^{(l)}\right\Vert_r \thicksim \frac{1}{n},\, n\in \mathbb{N}
\end{eqnarray*}
in $\overline{D}_r,$ where the constants in the equivalence depend on $f, r, r_1$ and $l.$
\end{thm}
\begin{proof}
Denoting by $\Gamma$ the circle of radius $r_1>r$ and centre $0,$ by Cauchy's formula, it follows that $\forall \,\, |z|\leq r,\mu\in \Gamma$ with $|\mu-z|\geq r_1-r,$ and $n\in \mathbb{N},$ we have
\begin{eqnarray*}
|G_n^{\alpha (l)}(f;z)-f^{(l)}(z)| &\leq& \frac{l!}{2\pi} \int_\Gamma \frac{|G_n^{\alpha}(f;\mu)-f(\mu)|}{|\mu-z|^{l+1}} |d\mu|\\
&\leq& \frac{l!}{2\pi}* \frac{2\pi r_1}{(r_1-r)^{l+1}} * \|G_n^{\alpha}(f)-f\|_{r_1}\leq \frac{C_{r_1}(f) \,l!\, r_1}{n\,(r_1-r)^{l+1}}.
\end{eqnarray*}
From Theorem $\ref{t2},$ we obtain
\begin{eqnarray*}
\lim_{n\rightarrow \infty} n(G_n^{\alpha (l)} (f;z)-f^{(l)}(z))= \left(z(1-z) f^{\prime\prime}(z)\right)^{(l)}, \,\, \mbox{uniformly in}\,\,  \overline{D}_r.
\end{eqnarray*}
So, $\exists$ constants $0<M_1, \, M_2<\infty$ independent of $n$ such that
\begin{eqnarray*}
M_1\leq n \|G_n^{\alpha (l)}-f^{(l)}\|_r \leq M_2,\,\, \Rightarrow \frac{M_1}{n}\leq \|G_n^{\alpha (l)}-f^{(l)}\|_r \leq \frac{M_2}{n},\,\, \mbox{holds in}\,\, \overline{D}_r.
\end{eqnarray*}
This completes the proof.
\end{proof}

\section{Further application}
As applications of our complex operators $(\ref{n1})$, one can study some shape preserving properties. Thus, reasoning exactly as it was done in the case of complex $\alpha-$Bernstein polynomials in \cite{NC}, one can prove that beginning with an index, the summation-integral operators $G_n^\alpha(f;z)$ in this paper approximate the analytic functions, preserving in addition, the geometric properties of star likeness, convexity, univalence and spirallikeness in a certain disk.

\section{Data availability statement}
Data sharing not applicable to this manuscript as no datasets were generated or analysed during the current study.


\begin{thebibliography}{99}
\bibitem{AMM} T. Acar, A. M. Acu and N. Manav, Approximation of functions by genuine Bernstein-Durrmeyer type operators, J. Math. Inequal., 12 (4) (2018) 975-987.

\bibitem{AG1} G. A. Anastassiou and S. G. Gal, Approximation by complex Bernstein-Schurer and Kantorovich-Schurer polynomials in compact disks, Comput. Math. Appl. 58 (4) (2009) 734-743.

\bibitem{AG2} G. A. Anastassiou and S. G. Gal, Approximation by complex Bernstein-Durrmeyer polynomials in compact disks, Mediterr. J. Math. 7 (4) (2010) 471-482.

\bibitem{VG2} D. C\'{a}rdenas-Morales, P. Garrancho and I. Rasa, Approximation properties of Bernstein-Durrmeyer type operators, Appl. Math. Comput., 232 (1) (2014) 1-8.

\bibitem{CTLX} X. Chena, J. Tana, Z. Liua and J. Xieb, Approximation of functions by a new family of generalized Bernstein operators, J. Math. Anal. Appl. 450, (2017) 244-261.

\bibitem{NC} N. \c{C}etin, Approximation and geometric properties of complex $\alpha-$Bernstein operator, Results Math., 74, (2019) 40.

\bibitem{SG} S. G. Gal, Approximation by complex Bernstein and convolution-type operators. World Scientific Publication Co, Singapore (2009).

\bibitem{GG} S. G. Gal and V. Gupta, Quantitative estimates for a new complex Durrmeyer operator in compact disks, Appl. Math. Comput., 218 (2011) 2944-2951.




%
%


\bibitem{SG4} S. G. Gal, Approximation by complex genuine Durrmeyer type polynomials in compact disks, Appl. Math. Comput. 217 (2010) 1913-1920.

\bibitem{GA} M. Goyal and P. N. Agrawal, Blending type approximation by complex Sz\'{a}sz-Durrmeyer-Chlodowsky operators in compact disks, Math. Slovaca 69 (5) (2019) 1-12.

\bibitem{GS} V. Gupta and D. Soyba\c{s}, Approximation by complex genuine hybrid operators, Appl. Math. Comput., 244 (2014) 526-532.

\bibitem{VG1} V. Gupta, Approximation properties by Bernstein-Durrmeyer type operators, Complex Anal. Oper. Theory, 7 (2013) 363-374.


\bibitem{NIM1} N. I. Mahmudov, Approximation by Bernstein-Durrmeyer-type operators in compact disks, Appl. Math. Lett. 24 (7) (2011) 1231-1238.

\bibitem{GGL} G. G. Lorentz, Bernstein Polynomials, 2nd ed., Chelsea Publ., New York, 1986.

\end{thebibliography}
\end{document}